\documentclass[11pt]{amsart}
\usepackage{amsmath,amssymb}
\usepackage{a4wide}

\newcommand{\Z}{{\mathbb Z}}
\newcommand{\N}{{\mathbb N}}
 
\newcommand{\R}{{\mathbb R}}
\newcommand{\Per}{\textrm{Per}}

\newtheorem{theorem}{Theorem}[section]
\newtheorem{corollary}[theorem]{Corollary}
\newtheorem{lemma}[theorem]{Lemma}
\newtheorem{proposition}[theorem]{Proposition}

\newtheorem{remark}[theorem]{Remark}

\newcommand{\eps}{\varepsilon}

\title{The isoperimetric problem for nonlocal perimeters}

\author[A. Cesaroni, M. Novaga]{}

\subjclass{ 53A10, 49Q20, 35R11.}
  \keywords{Fractional perimeter, nonlocal isoperimetric inequality, Poincar\'e inequality}
 \email{annalisa.cesaroni@unipd.it}
 \email{matteo.novaga@unipi.it}
  
\thanks{The authors were supported the Fondazione CaRiPaRo
Project ``Nonlinear Partial Differential Equations:
Asymptotic Problems and Mean-Field Games", Project PRA 2017 of the University of Pisa ``Problemi di ottimizzazione e di evoluzione in ambito variazionale",
 the INdAM-GNAMPA project ``Tecniche EDP, dinamiche e probabilistiche per lo studio di problemi asintotici"}

\begin{document}
\maketitle 

\centerline{\scshape Annalisa Cesaroni}
\medskip
{\footnotesize
 \centerline{Department of Statistical Sciences}
   \centerline{University of Padova}
   \centerline{Via Cesare Battisti 141, 35121 Padova, Italy}
} 

\medskip

\centerline{\scshape Matteo Novaga}
\medskip
{\footnotesize
 \centerline{ Department of Mathematics}
   \centerline{University of Pisa}
   \centerline{Largo Bruno Pontecorvo 5, 56127 Pisa, Italy}
}
\begin{abstract}
We consider a class of  nonlocal generalized perimeters which includes fractional perimeters and Riesz type potentials.
We prove a general isoperimetric inequality for such functionals, and we discuss some applications. In particular we prove existence of an isoperimetric profile, under suitable assumptions on the interaction kernel. 
\end{abstract}

\section{Introduction} 
In this paper we consider a family of geometric functionals, which in particular contains the fractional isotropic and anisotropic perimeter. More precisely, 
we define the  following energy defined on measurable subsets $E\subset\R^N$: 
\begin{equation}\label{f} 
\Per_K(E):= \int_E\int_{\R^N\setminus E} K(x-y)dxdy=\frac{1}{2} \int_{\R^N}\!\!\int_{\R^N} |\chi_E(x)-\chi_E(y)|K(x-y) dxdy
\end{equation}
where the kernel $K:\R^N\to [0, +\infty)$ satisfies the following assumptions: 
\begin{eqnarray}
\label{sim} & K(x)=K(-x)  \\
 \label{int} &\min(|x|, 1)\,K(x)\in L^1(\R^N).
 \end{eqnarray} 

The functional \eqref{f} measures the interaction between points in $E$ and in $\R^N\setminus E$, weighted by the kernel $K$.  

Note that it is not restrictive to assume \eqref{sim} since 
$\Per_{  K}(E)=\Per_{\widetilde K}(E)$ for every $E$, where 
$\widetilde K(x):=(K(x)+K(-x))/2$. 
Notice also that, if $K\in L^1(\R^N)$, then for every $E$ with $|E|<\infty$, we have
\begin{equation}\label{e} \Per_K(E)= |E|\|K\|_{L^1(\R^N)} -\int_E\int_E K(x-y)dxdy.\end{equation}

In the first part of the paper we deal with isoperimetric inequalities for such functionals. The main result is the following (see Corollary \ref{coro}): 
if  $K(x)\ge \mu \chi_{B_r}(x)$ for some constants $\mu>0$ and $r>0$, then 
for all measurable  sets $E$ there holds
\[\Per_K(E) \ge
\min ( g(|E|),  g(|\R^N\setminus E|)),
\] 
where $g(m):=\Per_{K^\star}(B_m)$, with $K^\star$ 
the symmetric decreasing rearrangment of $K$, and $B_m$ the ball 
with volume $m$ centered at $0$. 
We discuss some property of the function $g$ and we provide a Poincar\'e type inequality
(see Proposition \ref{em}).

We  recall that, in the case of fractional perimeters,
sharp quantitative isoperimetric inequality, uniform with respect to the fractional exponent bounded away from $0$, have been obtained in \cite{ffmmm} (see also \cite{l} for an anisotropic version), whereas  Poincar\'e type inequalities have been discussed in \cite{m}.

An interesting related question is understanding which conditions on $K$ imply the  compact embedding of the functions
with bounded energy $J_K$ into $L^p$ spaces, for some $p\geq 1$. 

\smallskip

In the second part of the paper, we consider the isoperimetric problem 
\begin{equation}\label{isopb}
\min_{E:\ |E|=m} \Per_K(E),
\end{equation}
for a fixed volume $m>0$.

In the case of the  fractional perimeter, 
the existence of isoperimetric sets solving \eqref{isopb} has been studied in
\cite{crs, ffmmm} (see also \cite{cn} where a bulk term is added to the energy),
where it is shown that balls are the unique minimizers of the fractional perimeter among sets with the same volume.
In the general case, the same result holds if the kernel $K$ is a radially symmetric decreasing function, as a straightforward 
consequence of the Riesz rearrangement inequality \cite{riesz}. 
So, we focus on the case in which $K$ is not radially symmetric and decreasing.
We provide an existence result of  minimizers of the relaxed problem associated to \eqref{isopb} under the additional assumption that  $K\in L^1(\R^N)$ (see Theorem \ref{existence}). 
The proof is based on a concentration compactness type argument. 
Finally, we show that if $K$ has maximum at the origin (in an appropriate sense, see condition \eqref{pos}), then every minimizer of the relaxed problem is actually the characteristic function of a compact set (see Theorem \ref{existenceiso}). 

We are left with  the  open problem  of extending the existence result to more general interaction kernels satisfying only \eqref{int}. 

Another interesting problem is to consider  kernels which are just Radon measures on $\R^N$. In this case we don't expect in general compactness of minimizers.

\subsection*{Notation}
We  denote by $B_m(x)$ the ball  centered at $x$ with volume $m$,  that is, the ball with radius $r=m^{\frac{1}{N}}\omega_N^{-\frac{1}{N}}$, and
by $B_m$ for the ball centered at $0$ with volume $m$. We also denote by $B(x,r)$ the ball of center $x$ and radius $r$. 

For every measurable set $E\subseteq \R^N$, $\chi_E$ denotes the characteristic function of $E$, that is
the function which is $1$ on $E$ and $0$ outside.   

We recall that given a set $E$ with $|E|<\infty$, its symmetric rearrangement $E^\star$ is the ball $B_{|E|} $ that is the ball centered at $0$ with volume  $|E|$. 
Moreover the symmetric decreasing rearrangement of a nonnegative measurable  function $h$ with level sets of finite measure is defined as 
\[h^\star(x)= \int_0^{+\infty} \chi_{\{h>t\}^\star}(x) dt. \]
Note that if $h$ is radially symmetric and decreasing, then $h=h^\star$. 
Moreover, $h\in L^p(\R^n)$ if and only if $h^\star \in L^p(\R^n)$ with $\|h\|_{L^p(\R^n)}=\|h^\star\|_{L^p(\R^n)}$, for all $p\ge 1$. 

\section{Generalized fractional perimeters}
In this section we discuss some properties of the $K$ perimeters.


\begin{remark}\upshape \label{remarkfinit}  
Condition \eqref{int} implies that if $E$ is a  set with $|E|<\infty$ and $\mathcal{H}^{N-1}(\partial E)<\infty$,  then $\Per_K(E)<\infty$ (see  \cite[Remark 1.4]{csv}).
Indeed 
\begin{multline*} \Per_K(E)= \int_E\int_{\R^N\setminus E} K(x-y)dxdy\\=\int_{\R^N} |(E+x)\cap (\R^N\setminus E)| K(x)dx\leq C \int_{\R^N}( |x|\wedge 1)K(x)dx\end{multline*}
where $C$ is a constant which depends on $E$.  
\end{remark} 
 
\begin{proposition} \label{propper} 
The following properties hold:
\begin{enumerate}
\item $\Per_K(E)=\Per_K(\R^N\setminus E)$ and
\begin{equation}\label{sub} \Per_K(E\cap F)+\Per_K(E\cup F)\leq \Per_K(E)+\Per_K(F).\end{equation}
\item $E\to \Per_K(E)$ is lower semicontinuous with respect to the $L^1_{loc}$-convergence. 
\end{enumerate} 
\end{proposition}

\begin{proof}
We start by proving 1. 
The first equality is a direct consequence of the definition of $\Per_K$.
In order to prove \eqref{sub}, we
observe that 
 \begin{multline*} \int_{E\cup F}\int_{\R^N\setminus (E\cup F)}\\= \int_{E}\int_{\R^N\setminus E} +\int_{F}\int_{\R^N\setminus F}-\int_{E\cap F}\int_{\R^N\setminus (E\cup F)} -
 \int_E \int_{F\setminus(E\cap F)}-\int_F\int_{E\setminus (E\cap F)}. \end{multline*}
 and 
 \begin{multline*} \int_{E\cap F}\int_{\R^N\setminus (E\cap F)}\\= \int_{E\cap F}\int_{\R^N\setminus (E\cup F)}+\int_{E\cap F}\int_{E\setminus (F\cap E)}+\int_{E\cap F}\int_{F\setminus (F\cap E)}. \end{multline*}
 Therefore 
 \begin{multline*}\Per_K(E\cap F)+\Per_K(E\cup F)=\Per_K(E)+\Per_K(F)\\
 -2\int_{E\setminus (E\cap F)}\int_{F\setminus (E\cap F)} K(x-y)dxdy, \end{multline*}
which gives \eqref{sub}.
 
 The proof of 2. 
is a consequence of Fatou lemma, observing that 
$\Per_K(E)=\int_{\R^N}\!\!\int_{\R^N} |\chi_E(x)-\chi_E(y)|K(x-y)dxdy$. 
\end{proof} 

\subsection{Examples} 
A first class of examples is given by the kernels $K(x)$ which satisfies 
\[  \lambda \frac{1}{|x|^{N+s}}\leq K(x)\leq \Lambda \frac{1}{|x|^{N+s}},\] 
for some $s\in (0,1)$ and $0<\lambda\leq \Lambda$. This class includes the fractional perimeters, and its  inhomogeneous and anisotropic versions. 

The fractional perimeter, which has been introduced in \cite{m} and further developed 
in \cite{crs}, is defined as  \begin{equation}\label{fractionalp}
P_s(E):=\int_{E}\int_{\R^N\setminus E} \frac{1}{|x-y|^{N+s}}\,dx\,dy,  
\end{equation} 
for $s\in (0,1)$. It is also possible to substitute the kernel  $ \frac{1}{ |x - y|^{N+s}}$ with more general heterogeneous, isotropic kernels of the type
\[K(x)= \frac{a(x)}{ |x |^{N+s}}\] 
where $a:\R^N \to (0, +\infty)$ is a measurable function such that $0<\lambda\leq a(x)\leq \Lambda$. 
The  anisoptropic  fractional perimeters have been defined in \cite{l} as follows:
let $B\subseteq \R^N$ be a convex set which is symmetric with respect to the origin and let 
$|\cdot|_B$ the norm in $\R^N$  with unitary  ball $B$, then we define 
\begin{equation}\label{fractionalpaniso}
P_{s,B}(E):=\int_{E}\int_{\R^N\setminus E} \frac{1}{|x-y|_B^{N+s}}\,dx\,dy.
\end{equation} 

Another class of examples, relevant for this paper, is given by the kernels $K(x)\in L^1(\R^N)$,
for which the representation formula \eqref{e} holds.

\subsection{Coarea formula} 
We introduce the following functional on functions $u\in L^1_{loc}(\R^N)$:
\begin{equation}\label{funz1}
J_K(u)=\frac 12 \int_{\R^N}\!\!\int_{\R^N} |u(x)-u(y)|K(x-y)dxdy.
\end{equation} 
Note that $J_K(\chi_E)=\Per_K(E)$ for all measurable $E\subset \R^N$.

We provide a coarea formula, linking  the functional $\Per_K$ to $J_K$.
\begin{proposition}[Coarea formula]   
The following formula holds
\begin{equation}\label{coarea} J_K(u)= \int_{-\infty}^{+\infty} \Per_K(\{u>s\})ds.\end{equation} 

\end{proposition} 
\begin{proof} 
First of all we observe that, for every measurable function $u$, 
\[|u(x)-u(y)|=\int_{-\infty}^{+\infty} |\chi_{\{u>s\}}(x)-\chi_{\{u>s\}}(y)|ds.\]
Moreover for every $s\in\R$
\[ |\chi_{\{u>s\}}(x)-\chi_{\{u>s\}}(y)|=\chi_{\{u>s\}}(x)\chi_{\R^N\setminus\{u>s\}}(y)+\chi_{\{u>s\}}(y)\chi_{\R^N\setminus\{u>s\}}(x).
\]
Therefore we get, recalling \eqref{sim},  and using Tonelli theorem,
\begin{eqnarray*}2J_K(u)&=& \int_{\R^N}\!\!\int_{\R^N} |u(x)-u(y)|K(x-y)dxdy\\&=&2\int_{\R^N}\!\!\int_{\R^N}\int_{-\infty}^{+\infty} \chi_{\{u>s\}}(x)\chi_{\R^N\setminus\{u>s\}}(y)K(x-y)ds dxdy
\\ &=& 2 \int_{-\infty}^{+\infty} \int_{\{u>s\}}\int_{\R^N\setminus \{u>s\}} K(x-y)dxdyds\\ &=&   2 \int_{-\infty}^{+\infty} \Per_K(\{u>s\})ds.\end{eqnarray*}
\end{proof} 

\section{Isoperimetric inequality}
In this section we prove an isoperimetric inequality for generalized nonlocal perimeters. 
\begin{proposition}\label{isoperprop}  
For every measurable set  $E\subseteq \R^N$ such that $|E|<\infty$,  there holds
\[\Per_K(E)\geq \Per_{K^*}(B_{|E|}), \]
where $K^*$ is the symmetric decreasing rearrangement of $K$.

In particular, if $K$ is radially symmetric and decreasing, then 
 \[\Per_K(E)\geq \Per_{K}(B_{|E|}). \]
 Moreover, equality holds if and only if $E$ is a translated of $B_{|E|}$.
\end{proposition} 
\begin{proof} 
First of all we consider the case in which $K\in L^1(\R^N)$. 
Note that $(\chi_E)^\star=\chi_{B_{|E|}}$. 
By Riesz rearrangement inequality \cite{riesz}, 
we get that
\begin{multline*} \int_E\int_E K(x-y)dxdy=\int_{\R^N} \chi_E(x) (\chi_E* K)(x)dx\\ 
\leq \int_{\R^N}\chi_{B_{|E|}}(x) (\chi_{B_{|E|}}* K^\star)(x)dx=\int_{B_{|E|}}\int_{B_{|E|}} K^\star(x-y).  \end{multline*}
So, recalling \eqref{e} we get the conclusion.

Finally, if $K=K^\star$, we have that equality in the Riesz rearrangement inequality holds if and only if 
$\chi_E$ is equal, up to translation, to its symmetric-decreasing rearrangement, therefore if and only if $E$ is equal, up to translation, to $B_{|E|}$. 

Now if $K\not\in L^1(\R^N)$, we define $K_\eps(x)=K(x)\wedge \frac{1}{\eps}$. Then $K_\eps\in L^1(\R^N)$ and $K_\eps$ converges to $K$ monotonically increasing. 
 Note that \[K_\eps^\star(x)=\int_0^{\frac{1}{\eps}} \chi_{\{K >t\}^\star}(x) dt.\] 
 So as $\eps\to 0$, also  $K_\eps^\star\to K^\star$ monotonically increasing.
Therefore 
by the monotone convergence theorem if $E$ is a measurable set we get that 
\[\lim_{\eps\to 0} \Per_{K_\eps}(E)=\Per_K(E)\qquad \lim_{\eps\to 0} \Per_{K_\eps^\star}(E)=\Per_{K^\star}(E) . \]

By the previous argument we get $\Per_{K_\eps}(E)\geq \Per_{K_\eps^\star}(B_{|E|})$.
So, we conclude sending $\eps\to 0$. 
 \end{proof} 

For every $m\geq 0$, we define  
\[g(m):= \Per_{K^\star}(B_m)\]  where we recall that $B_m$ is the ball centered at $0$ with volume $m$. 
We also set $g(+\infty)=+\infty$. 

We provide some estimates on the function $g$. 

\begin{lemma} \label{propg} \ \ \ 
\begin{enumerate}

\item If $ K\not\in L^1(\R^N)$, then \begin{equation}\label{notl1}  
\lim_{m\to 0^+}\frac{g(m)}{m}=+\infty.\end{equation}
\item If  $K\in L^1(\R^N)$, then 
\begin{equation}\label{l1} g(m)\leq  \|K\|_{L^1} m,\qquad 
\lim_{m\to 0^+}\frac{g(m)}{m}= \|K\|_{L^1}.\end{equation}
\end{enumerate} 
\end{lemma} 

\begin{proof}
%
%
 
 For $x\in\R^N$ and $m>0$ we define 
\[  K_m^\star(x) = \frac{1}{m} \int_{B_m(x)} K^\star(y)dy.\]
We have
\begin{eqnarray*}
\int_{\R^N} K^\star_m(x)dx&=&\frac{1}{m} \int_{\R^N} \int_{B_m} K^\star(x+y)dydx
\\
&=&
\frac{1}{m} \int_{B_m} \int_{\R^N} K^\star(x)dxdy = 
 \int_{\R^N} K^\star(x)dx.\end{eqnarray*}
In particular   $K^\star \in L^1(\R^N)$  if and only if $K_m^\star\in L^1(\R^N)$ and
$\|K^\star\|_{L^1}=\|K_m^\star\|_{L^1}$. We recall  that $\|K^\star\|_{L^1}=\|K\|_{L^1}$.

We then compute
\[\Per_{K^\star}(B_m)=\int_{B_m}\int_{\R^N\setminus B_m} K^\star(x-y)dxdy
= m \int_{\R^N\setminus B_m} K^\star_m(x)dx,\]
which gives  \eqref{notl1} and \eqref{l1}, sending $m\to 0$.
\end{proof}

\begin{proposition}\label{probono}
Assume that the kernel $K$ satisfies the following condition:
\begin{equation}\label{kappazero}
\text{there exist $\mu,r>0$ such that }K(x)\ge\mu \ \text{for all }x\in B(0,r).
\end{equation}
Let $E$ be a measurable set such that $\Per_K(E)<\infty$, then $|E|<\infty$ or 
$|\R^N\setminus E|<\infty$.
\end{proposition} 

\begin{proof}
Let $\{Q_i\}_{i\in \mathbb N}$ be a partition of $\R^N$ made of cubes of sidelength $r/\sqrt{n}$, where $r$ is as in \eqref{kappazero}.
Note that for all $x,y\in Q_i$ we have $K(x-y)\ge \mu$.

Assume by contradiction that $|E|=|\R^N\setminus E|=\infty$. Then three possible cases may verify: either there exists $\delta>0$ such that 
$ \limsup_{i}|E\cap Q_i|\leq (1-\delta) |Q_i|$ or $ \limsup_{i}|  Q_i\setminus E|\leq (1-\delta) |Q_i|$ or there exist two  subsequences  $Q_{i_n},  Q_{j_n}$ such that  $\lim_{n}|Q_{i_n}\cap  E|= \lim_{n}| Q_{j_n}\setminus E|= |Q_i|$. 

{\it Case 1}: assume there exists $\delta>0$ such that $ \limsup_{i}|E\cap Q_i|\leq (1-\delta) |Q_i|$.
So there exists $i_0>0$ such that for all $i\geq i_0$, $|Q_i\setminus E|\geq \delta/2 |Q_i|$. Then we get
\[\Per_K(E)\geq \sum_i \int_{Q_i\cap E}\int_{Q_i\setminus E} K(x-y)dxdy\geq \mu \sum_{i\geq i_0}  |Q_i\cap E| |Q_i\setminus E| \geq \mu \frac{\delta}{2} 
\sum_{i\geq i_0}|E\cap Q_i|. \]
This implies, recalling that $\Per_K(E)<\infty$,  that $\sum_{i}|E\cap Q_i|<\infty$, which is in contradiction with $|E|=\infty$. 

{\it Case 2}: assume there exists $\delta>0$ such that $ \limsup_{i}|E\cap Q_i|\leq (1-\delta) |Q_i|$. 
Then the argument is the same as in {\it Case 1}, substituting $E$ with $\R^N\setminus E$.

{\it Case 3}: assume  that here exist two   subsequences  $Q_{i_n},  Q_{j_n}$ such that  $\lim_{n}|Q_{i_n}\cap  E|= \lim_{n}| Q_{j_n}\setminus E|= |Q_i|$. 
 Therefore for $\delta>0$ there exists $i_0$ such that for all $j_n, i_n \geq i_0$, we get $| Q_{i_n}\cap  E|, | Q_{j_n}\setminus E|>(1-\delta)|Q_i|$. 
By continuity we get that there exists a subsequence $\bar Q_{i}$  such that  $|\bar Q_{i}\cap E|, |\bar Q_{i}\setminus E|\geq \delta |\bar  Q_i|$, 
So,
\begin{multline*} \Per_K(E)\geq \sum_{i}  \int_{\bar Q_{i}\cap E}\int_{\bar Q_{i}\setminus E} K(x-y)dxdy\\ 
\geq \mu \sum_{i}  |\bar Q_i\cap E| |\bar Q_i\setminus E| \geq \mu \delta^2\sum_{i} |\bar Q_i|^2 =\infty\end{multline*} 
giving a contradiction to $\Per_K(E)<\infty$. 

\smallskip

As a consequence, neither of the three cases can arise, which implies that either $|E|$ or $|R^N\setminus E|$ are finite. 

\end{proof}


From Propositions \ref{isoperprop} and \ref{probono}
we immediately get the following result:

\begin{corollary} \label{coro} 
Assume that $K$ satisfies condition \eqref{kappazero}. Then, for all measurable  sets $E$ there holds
\[\Per_K(E) \ge
\min ( g(|E|),  g(|\R^N\setminus E|)).  
\] 
\end{corollary}

\section{Poincar\'e inequality} 
\begin{proposition} \label{em} 
Assume that the kernel $K$ satisfies  \eqref{kappazero} and that there exist $k\geq 1$  and a constant $C$ depending on $k$, $N$ such that 
\begin{equation}\label{isoperi1} 
g(m)\ge \frac{m^k}{C} \qquad \forall m>0.
\end{equation}
Then, for all  $u\in L^1_{loc}$ with  $J_K(u)<\infty$, there holds  
 \begin{equation}\label{poincare1} \|u-m(u)\|_{L^{k}(\R^N)}\leq C J_K(u),\end{equation}
where \begin{equation}\label{median} 
m(u)=\inf\{s\  |\   |\{u(x)>s\}|<\infty\}\in \R. \end{equation}
Moreover if $u\in L^1(\R^N)$, then $m(u)=0$.
\end{proposition}  

\begin{proof}
The argument is similar to the one  in  \cite[Theorem 3.47]{afp}. 
First of all we observe that, since $J_K(u)<\infty$,  by the coarea formula \eqref{coarea} 
the set $S$ of $s\in\R$ such that $\Per_K(\{x\ | u(x)>s\})<\infty$ is dense in $\R$.  
So by Proposition \ref{probono} for every $s\in S$, either  $|\{u(x)>s\}|<\infty$ or $ |\{u(x)\leq s\}|<\infty$.
Note that if $s>m(u)$, then there exists $t\in (m(u),s)$ such that $ |\{u(x)>t\}|<\infty$ and then $ |\{u(x)>s\}|<\infty$. Analogously, if $s<m(u)$, then
$ |\{u(x)\leq t\}|<\infty$. 
Moreover $m(u)\in\R$. Indeed, if by contradiction this were not true, and e.g. $m(u)=-\infty$ (the other case being similar), we would  get that 
$|\{u(x)> t\}|<\infty$ for every $t\in\R$.  By the coarea formula 
$$
\int_{-n-1}^{-n} \Per_K(\{u(x)> t\})dt\le \frac 12 J_K(u) \qquad \forall n\in\mathbb N,
$$
so  there exist $r>0$
 and $t_n\in [-n-1,-n]$ such that $\Per_K(\{u(x)> t_n\}\leq r$. By the isoperimetric inequality $|\{u(x)> t_n\}|\leq (Cr)^k$,
but this is contradiction with the fact that $\{u(x)> t_n\}\to \R^N$ as $n\to +\infty$.

Finally, if $u\in L^1(\R^N)$, $m(u)=0$. Indeed, by Chebychev inequality for every $s>0$, we have that 
\[ |\{x\ | u(x)\leq -s\}|+|\{x\ | u(x)> s\}|\leq |\{x\ | |u(x)|\geq s\}|\leq \frac{1}{s} \int_{\R^N} |u|dx<+\infty.\]

We denote by $u^+=\max (u-m(u), 0)$ the positive part of $u-m(u)$. Then, by definition of $m(u)$, we get that $|\{x\ | u^+(x)>s\}|<\infty$ for all $s>0$. 

After a change of variable, we get 
\begin{multline}\label{uno} \int_{\R^N} (u^+)^{k}dx=\int_0^{+\infty} |\{x\ | u^+(x)>t^{\frac{1}{k}}\}| dt\\
=k\int_0^{+\infty} |\{x\ | u^+(x)>s\}| s^{k-1}ds.
\end{multline} 
In \cite[Lemma 3.48]{afp})  it is shown that, if $f:(0, +\infty)\to [0, +\infty)$ is  decreasing and $k\geq 1$, then 
\[k\int_0^T f(s) s^{k-1} ds\leq \left(\int_0^T f(s)^{\frac{1}{k}}ds \right)^k\qquad \forall T>0. \]
So, we apply this inequality to  $f(s)= |\{x\ | u^+(x)>s\}|$. 
This gives, by  recalling  the definition of $m(u)$ and using isoperimetric inequality \eqref{isoperi1}, 
\begin{multline} \label{due} k\int_0^{+\infty} |\{x\ | u^+(x)>s\} | s^{k-1}ds\leq  \left( \int_0^{+\infty} |\{x\ | u^+(x)>s\}|^{\frac{1}{k}}ds\right)^{k}
\\ \leq   \left( C \int_0^{+\infty} J_K(\{x\ | u(x)>s\}) ds\right)^{k}.
\end{multline} 
Therefore, putting together \eqref{uno} and \eqref{due} and recalling the coarea formula \eqref{coarea}, we get 
\[\left(\int_{\R^N} (u^+)^{k}dx\right)^{\frac{1}{k}}\leq C J_K(u). \] 
Repeating the same argument for the negative part of $u-m(u)$, i.e.  $u^-=-\min (u-m(u),0)$, we conclude. Indeed, again by   definition of $m(u)$, we get that 
for all $s>0$, $|\{x\ | u^-(x)>s\}|<\infty$. 
 \end{proof} 

\begin{remark}\upshape
If $u\in L^1(\R^N)$ assumption \eqref{kappazero} is not needed, indeed for all $s\in \R$ with $s\ne 0$ either 
$|\{u(x)>s\}|<\infty$ or $ |\{u(x)\leq s\}|<+\infty$, so that it is not necessary to use
Proposition \ref{probono}.
\end{remark}

 \section{Existence of an isoperimetric profile}

In this section we show the existence of an isoperimetric profile, that is, a solution to Problem \eqref{isopb}, under suitable assumptions on the kernel $K$.
First of all we will assume throughout  this section that $K\not\equiv 0$ and 
\begin{equation}\label{kinl1} K\in L^1(\R^N). \end{equation} 

We observe that, if $K=K^\star$, then by Proposition \ref{isoperprop} we know that the ball of volume $m$
is the unique minimizer of \eqref{isopb}, up to translations.

In order to get existence of minimizers, we first consider a relaxed version of the perimeter functional, obtained by extending it to general densities functions.
More precisely, we define the new energy as  follows: given $f:\R^N\to [0,1]$, with $f\in L^1(\R^N)$, we let 
\begin{equation}\label{relaxed} 
\mathcal{P}_K(f):=\int_{\R^N}\!\!\int_{\R^N} f(x)[1-f(y)] K(x-y)dxdy.
\end{equation} 
Note that  $\mathcal{P}_K(\chi_E)= \Per_K(E)$ and the constraint $0\leq f\leq 1$  
is inherited by the original problem, naturally arising from the relaxation procedure. 

Note that  the previous energy can be written as 
\begin{equation}\label{relaxed2}
\mathcal{P}_K(f)=\|f\|_{L^1}\|K\|_{L^1}-\int_{\R^N}\!\!\int_{\R^N} f(x)f(y) K(x-y)dxdy.
\end{equation}

The relaxed version of the isoperimetric problem \eqref{isopb} can be restated as follows.  Given $m\geq 0$, we consider 
\begin{equation} \label{min}
\inf_{f\in \mathcal{A}_m} \mathcal{P}_K(f), 
\end{equation} 
where  the set of admissible functions is defined as 
\[\mathcal{A}_m=\left\{ f\in L^1(\R^N, [0,1]), \\ \int_{\R^N} f(x)dx=m\right\}.\]

Notice also that 
 \[\liminf_n \Per_K(E_n)= \mathcal{P}_K(f)\]
where the liminf is taken over all sequences   $E_n$  with $|E_n|=m$, such that  
$\chi_{E_n}\stackrel{*}{\rightharpoonup} f$ weakly$^*$ in $L^\infty$. 

Due to \eqref{relaxed2}, the minimization problem \eqref{min} is equivalent to the maximization problem
\begin{equation} \label{maxrel}
\sup_{f\in \mathcal{A}_m}\int_{\R^N}\!\!\int_{\R^N} f(x)f(y)K(x-y)dxdy. 
\end{equation} 

We now show monotonicity and subadditivity  of the energy in \eqref{maxrel} with respect to $m$.
\begin{lemma}\label{suba}
\ \ \  \  
\begin{itemize}
\item[i)] If $m_1>m_2>0$ then
\[\sup_{f\in \mathcal{A}_{m_1}}\int_{\R^N}\!\!\int_{\R^N} f(x)f(y)K(x-y)dxdy\geq \sup_{f\in \mathcal{A}_{m_2}}\int_{\R^N}\!\!\int_{\R^N} f(x)f(y)K(x-y)dxdy. \]
\item[ii)] If $\sum_{i=1}^l m_i=m$ then 
\[\sum_{i=1}^l \sup_{f_i\in \mathcal{A}_{m_i}}\int_{\R^N}\!\!\int_{\R^N} f_i(x)f_i(y)K(x-y)dxdy\leq\sup_{f\in \mathcal{A}_m}\int_{\R^N}\!\!\int_{\R^N} f(x)f(y)K(x-y)dxdy.\]
Moreover, if the equality holds in the above inequality and if the supremum in \eqref{maxrel}
is attained for all volumes $m_i$'s, then $m_i=0$ for all $i$'s except one. 
\end{itemize} 
\end{lemma} 

\begin{proof}
$i)$ Let $f\in \mathcal{A}_{m_2}$ and let $\eps>0$. 
Let $E\subseteq \R^N$ such that $f(x)<1-\eps$ in $E$ and $|E|=(m_1-m_2)/\eps$. We define 
\[ \tilde f(x):=f(x)+\eps\chi_E(x)\qquad x\in \R^N.\]
We have $\tilde f\in \mathcal{A}_{m_1}$.
Moreover, since $K\geq 0$, 
\begin{multline*} 
\int_{\R^N}\!\!\int_{\R^N} \tilde f(x)\tilde f(y)K(x-y)dxdy= \int_{\R^N}\!\!\int_{\R^N} f(x)f(y)K(x-y)dxdy\\+4\eps \int_{\R^N}\int_{E} f(x)K(x-y)dxdy+\eps^2\int_{E}\int_{E}K(x-y)dxdy\\
\ge\int_{\R^N}\!\!\int_{\R^N} f(x)f(y)K(x-y)dxdy.\end{multline*} 

$ii)$ We consider the case $l=2$, as the case $l>2$ can be treated analogously. Let $f_i\in\mathcal{A}_{m_i}$, and let $\eps>0$ and $R>0$ such that $\int_{\R^N\setminus B(0,R)} f_i(x)dx\le
\eps$ for $i=1,2$.  Note that 
\[  \left|\int_{\R^N}\!\!\int_{\R^N} f_i(x) f_i(y)K(x-y)dxdy-\int_{B(0,R)}\int_{B(0,R)} f_i(x) f_i(y)K(x-y)dxdy\right|\leq 2\,\eps\, \|K\|_{L^1}.\]

 We fix $x_R\in \R^N$ such that $(f_1(x)\chi_{B(0, R)}(x))(f_2(x-x_R)\chi_{B(x_R,R)})=0$ for 
 a.e. $x$, and we let  $$f(x):=f_1(x)\chi_{B(0, R)}(x)+f_2(x-x_R)\chi_{B(x_R,R)}.$$ 
 Then, $f\in\mathcal{A}_{m'}$ for some $m'\in[m-2\eps,m]$.
Hence, by item $i)$, we get
\begin{multline}\label{don} 
\sup_{g\in \mathcal{A}_m}\int_{\R^N}\!\!\int_{\R^N} g(x)g(y)K(x-y)dxdy\geq 
\sup_{g\in \mathcal{A}_{m'}}\int_{\R^N}\!\!\int_{\R^N} g(x)g(y)K(x-y)dxdy\\ \geq  \int_{\R^N}\!\!\int_{\R^N} f(x) f(y)K(x-y)dxdy= \sum_{i=1}^2 \int_{B(0,R)}\int_{B(0,R)} f_i(x) f_i(y)K(x-y)dxdy\\
+\int_{B(0,R)}\int_{B(0, R)}f_1(x)f_2(y)K(x-y+x_R) dxdy\\\geq \sum_{i=1}^2 \int_{\R^N}\!\!\int_{\R^N} f_i(x) f_i(y)K(x-y)dxdy-4\,\eps\, \|K\|_{L^1},
\end{multline} 
from which we conclude by the arbitrariness of $\eps$.  

Note that, since $K\not\equiv 0$, we can always choose $x_R$ such that 
\[\int_{B(0,R)}\int_{B(0, R)}f_1(x)f_2(y)K(x-y+x_R) dxdy>0\] so that the last inequality in \eqref{don}  
is in fact a strict inequality. 
\end{proof} 

\subsection{The potential function} 
Given a function $f\in \mathcal{A}_m$, we can define the potential of $f$ as 
\begin{equation}\label{pot}
V(x):=\int_{\R^N} f(y)K(x-y)dy. \end{equation} 
In the following we give some properties of the potential $V$.
\begin{proposition}\label{propot}
\ \ \  \  
\begin{itemize}
\item[i)]  $V\in C(\R^N)\cap L^1(\R^N)\cap L^\infty(\R^N)$, with $0\leq V\leq \|K\|_{L^1}$ and $\|V\|_{L^1}=m\|K\|_{L^1}$.  
\item[ii)] $\lim_{|x|\to +\infty}V(x)=0$. 
\item[iii)] There exists $x\in \R^N$, density point of $f$,  such that $f(x)<1$  and $V(x)>0$.
\end{itemize} 
\end{proposition}  

\begin{proof} 
$i)$ By  definition $V\in L^1\cap L^\infty$, with $0\leq V\leq \|K\|_{L^1}$ and  $\|V\|_{L^1}=m \|K\|_{L^1}$.
Observe that since $f\in L^1$ and $K\in L^1$,  \[\lim_{h\to 0}\int_{\R^N} |f(y-x-h)-f(y-x)|K(y)dy =0\qquad \text{for every  $x$.}\] 
This implies that $V$ is  continuous.  

$ii)$ Let $\eps>0$ and let $A_\eps:=\{f>\eps\}$. We have
\begin{equation}\label{euno}
V(x)=\int_{\R^N} K(y)f(x-y)dy\le \eps \|K\|_{L^1} + \int_{x-A_\eps} K(y)dy.
\end{equation} 
For $R>0$ we have
\begin{equation}\label{edue}
\int_{x-A_\eps} K(y)dy =  \int_{x-(A_\eps\cap B(0,R))} K(y)dy +  \int_{x-(A_\eps\setminus B(0,R))} K(y)dy.
\end{equation} 
Since $K\in L^1(\R^N)$ and $\lim_{R\to\infty}|A_\eps\setminus B(0,R)|=0$, for $R$ sufficiently large we have 
$$
\int_{x-(A_\eps\setminus B(0,R))} K(y)dy\le \eps \|K\|_{L^1}, 
$$
which gives, recalling \eqref{edue},
\begin{eqnarray}\label{etre}
\int_{x-A_\eps} K(y)dy &\le&\int_{x-(A_\eps\cap B(0,R))} K(y)dy + \eps \|K\|_{L^1}
\\ \nonumber &\le& \int_{|y|\ge |x|-R} K(y)dy + \eps \|K\|_{L^1}\le 2\, \eps\, \|K\|_{L^1},
\end{eqnarray}
for $x$ large enough (depending on $R$).
The thesis now follows from \eqref{euno}, \eqref{etre} and the arbitrariness of $\eps$.
 
$iii)$ Since $K\not\equiv 0$, there exist $\delta>0$ and a bounded set $A$ with $|A|=k>0$ such that $0\not\in A$ and  $K(z)\geq \delta$ for a.e. $z\in A$. 
Let $x_0\in\R^N$ be a Lebesgue point of $f$ such that $f(x_0)>0$. Let  $y_0\in A$ be a point of density $1$ in $A$, such that $x_0+y_0$ is a density point of $f$ and we compute
\[V(x_0+y_0) \geq  \int_{A} f(x_0+y_0-z)K(z)dz\geq \delta \int_{A} f(x_0+y_0-z)dz>0.\]
If $f(x_0+y_0)<1$, we are done, $c\geq V(x_0+y_0)>0$. If, on the other hand, $f(x_0+y_0)=1$ for all points $y_0$ which are points of density $1$ of $A$ and density points for $f(x_0+\cdot)$, then we consider $y_1\in A$ be a point of density $1$ in $A$, such that $x_0+y_0+y_1$ is a density point of $f$ and we compute
\[V(x_0+y_0+y_1) \geq  \int_{A} f(x_0+y_0+y_1-z)K(z)dz\geq \delta \int_{A} f(x_0+y_0+y_1-z)dz>0.\]
Repeating this argument, we construct a sequence $y_n\in A$ of points of density $1$ such that $x_0+y_0+\dots+y_n$ is a density point of $f$, and such that $V(x_0+y_0+\dots+y_n)>0$. Note that for some $n\geq 0$, we get that $f(x_0+y_0+\dots+y_n)<1$. If it were not the case, we would get
\[m=\int_{\R^N} f(x)dx\geq \sum_{n} \int_{x_0+nA} f(x)dx =\sum_{n} |A|= +\infty\] 
which is impossible. 
\end{proof} 

\subsection{First and second variation}
We now compute the first and second variation of the energy in \eqref{min}.

\begin{lemma}\label{variation} 
Let   $f\in\mathcal{A}_m$ be a minimizer of \eqref{min}.
Let  $S:=\{x\ | f(x)=1\}$ and $N:=\{x\ | f(x)=0\}$.
\begin{itemize} 
\item[i)] 
For every  
$\psi,\phi\in  L^1(\R^N, [0,1])$  with $\int_{\R^N}\phi(x)dx=\int_{\R^N}\psi(x)dx$, and such that 
$\psi\equiv 0$ a.e. in $S$ and $\phi\equiv 0$ a.e. in $N$, the following holds 
\begin{equation}\label{fv1} \int_{\R^N}   (\psi(x)-\phi(x))V(x)dx \leq 0.\end{equation}

\item [ii)] There exists a constant $c> 0$ 
such that  \begin{equation}\label{fv}\begin{cases} V(x)\equiv c &  \text{ for every  }x\in \R^N\setminus(N\cup S)\\
V(x)\geq c &  \text{ for every  }x\in S\\ 
V(x)\leq c &  \text{ for every  }x\in N. \end{cases}\end{equation} 

\end{itemize} 
\end{lemma} 
\begin{proof} 

$i)$ We argue as in \cite[Lemma 1.2]{cdnp}.  
First observe that  for every $\lambda$, $\int_{\R^N} f+\lambda(\psi-\phi) dx=\int_{\R^N} f dx=m$. Moreover, for all $\lambda\in [0,1]$ and 
a.e. $x\in S\cup N$, we get that $f(x)+\lambda(\psi(x)-\phi(x))\in [0,1]$. We consider two sequences  
$\psi_\eps\to \psi$, $\phi_\eps\to \phi$ in $L^1$ such that
 $\int_{\R^N}\psi_\eps dx=\int_{\R^N}\phi_\eps dx=\int_{\R^N}\phi dx$ and such that $\psi_\eps(x)\equiv 0$ on the set $\{x | f(x)>1-\eps\}$ and $\phi_\eps\equiv 0$ 
 on the set $\{x\ | f(x)\leq \eps\}$. 
 So, choosing $\lambda>0$ sufficiently small (depending on $\eps$) we can show that $ f+\lambda(\psi_\eps-\phi_\eps)\in \mathcal{A}_m$. 
 By minimality of $f$ we get 
\[\frac{1}{\lambda} (\mathcal{P}_K(f+\lambda(\psi_\eps-\phi_\eps))- \mathcal{P}_K(f))\geq 0. \] So, sending $\lambda\to 0$ and recalling that $K$ is symmetric \eqref{sim}, we get 
 \[\int_{\R^N} (\psi_\eps(x)-\phi_\eps(x))(1-2f(y))K(x-y)dx \geq 0. \] So,  sending $\eps\to 0$ and recalling that $\int_{\R^N}(\phi(x)-\psi(x)) dx=0$, we conclude 
  \begin{multline*} 0\leq \int_{\R^N} (\psi(x)-\phi(x))(1-2f(y))K(x-y)dx\\= \|K\|_{L^1}\int_{\R^N}  (\psi(x)-\phi(x))dx- 2 \int_{\R^N} (\psi(x)-\phi(x))V(x)dx\\=- 2 \int_{\R^N} (\psi(x)-\phi(x))V(x)dx.\end{multline*} 
  
$ii)$ Choosing $\psi, \phi$ in \eqref{fv1}  such that $\psi=0=\phi$ on $S \cup N$, we can exchange the role of $\psi$ and $\phi$, and obtain that in $\R^N\setminus (N\cup S)$, $V$ has to be constant. So, there exists $c> 0$ (by Proposition \ref{propot} $iii)$) such that $V(x)\equiv c$ in $\R^N\setminus (N\cup S)$.

Choosing  $ \phi$ in \eqref{fv1}  such that $\phi=0 $ a.e. in  $S \cup N$, we get, since $\int_{\R^N}(\psi(x)-\phi(x))dx=0$, 
 $\int_{\R^N\setminus (N\cup S)} (\psi(x)-\phi(x))dx=-\int_{S} \psi(x) dx$. We compute  
 \begin{multline*}0\geq \int_{S} (\psi(x)-\phi(x))V(x)dx +\int_{N} (\psi(x)-\phi(x))V(x)dx +c \int_{\R^N\setminus (S\cup N)} (\psi(x)-\phi(x))dx\\
=\int_{N} \psi(x)V(x)dx  +c \int_{\R^N\setminus (S\cup N)} (\psi(x)-\phi(x))dx=\int_{N} \psi(x)(V(x)-c) dx    \end{multline*}
for all $\psi\in L^1(\R^N, [0,1])$ such that $\psi=0$ a.e. in $S$ and $\int_{\R^N}(\psi(x)-\phi(x))dx=0$.
This implies that $V\leq c$ in $N$. With an analogous argument, exchanging the role of $\psi$ and $\phi$, we get $V(x)\geq c$ in $S$. 

\end{proof} 

As immediate consequence of the first variation \eqref{fv} and of the properties of the potential $V$ 
we obtain that every minimizer  of \eqref{min}  has compact support.  

\begin{proposition}\label{compact} 
Every minimizer $f\in\mathcal{A}_m$  of \eqref{min}  has compact support.  
\end{proposition} 

\begin{proof}  By Proposition \ref{propot}  $\lim_{|x|\to +\infty}V(x)=0$. 
Hence we can find 
$R>0$ such that $0\leq V(x)<c$ for $|x|>R$, where $c>0$ is the constant appearing in \eqref{fv}. By \eqref{fv} this implies immediately that the support of $f$ is contained in $B_R(0)$. 
\end{proof}

We now consider the second variation of the functional.

\begin{lemma} \label{variation2} 
Let   $f\in\mathcal{A}_m$ be a minimizer of \eqref{min} and let $S,\,N$ be as in 
Lemma \ref{variation}. Then,
for every $\xi\in  L^1(\R^N, [-1,1])$  with $\int_{\R^N}\xi(x)dx=0$, and such that 
$\xi\equiv 0$ a.e. in $N\cup S$, it holds
 \begin{equation}\label{sv} 
 \int_{\R^N}\ \int_{\R^N}\xi(x)\xi(y)K(x-y)dxdy \leq 0. \end{equation} \end{lemma}
 
 \begin{proof} We argue as in \cite[Lemma 1.5]{cdnp}.   Reasoning as in the proof of Lemma \ref{variation}, item $i)$,  we can assume that
 there exists a sequence $\xi_\eps\to \xi$ in $L^1$ such that $\int_{\R^N} \xi_\eps(x)dx=0$ and $\xi_\eps\equiv 0$ a.e. in $\{x \ |\ f(x)\leq \eps \text{ or }  f(x)\geq 1-\eps\}$. So for $\lambda$
 sufficiently small $f+\lambda \xi_\eps\in\mathcal{A}_m$ and by  minimality we get
 \begin{multline*} 0\leq \mathcal{P}_K(f+\lambda\xi_\eps)- \mathcal{P}_K(f)\\ =\lambda \int_{\R^N\setminus (N\cup S)} (\|K\|_{L^1}-2V(x))
\xi_\eps (x)dx-\lambda^2\int_{\R^N} \int_{\R^N} \xi_\eps(x)\xi_\eps(y)K(x-y)dxdy.\end{multline*}
Recalling \eqref{fv} and the fact that $\int_{\R^N} \xi_\eps(x)dx=0$, we conclude the desired inequality by letting $\eps\to 0$. 
 \end{proof} 
 
\subsection{Existence of minimizers}

\begin{theorem} \label{existence} 
For every $m>0$ there exists at least one $f\in \mathcal{A}_m$ 
which solves the minimization problem \eqref{min}.
\end{theorem}

\begin{proof}
The proof is similar to that  in \cite[Theorem 1.9]{cdnp} (see also \cite{cn}), and
is based on a concentration compactness argument. 

Let $f_n\in\mathcal{A}_m$ be a minimizing sequence, and 
recall that the energy can be written as
\[\mathcal{P}_K(f_n)=m\|K\|_{L^1}-\int_{\R^N} \int_{\R^N} f_n(x)f_n(y) K(x-y)dxdy.\] 

We consider a partition of $\R^N$ in disjoint cubes: let $Q=[0,1]^n$ and let $Q^z:=z+Q$ for $z\in \Z^N$.
We fix $\eps>0$ small and divide the cubes in two subsets:  
\[I_{\eps,n}:=\{z\in\Z^N\ |\ \int_{Q^z} f_n(x)dx=:m_{n,z}\leq \eps\},\qquad A_{\eps,n}:=\cup_{z\in I_{\eps, n}} Q^z\] and \[J_{\eps,n}:=\{z\in\Z^N\ |\ \int_{Q^z} f_n(x)dx=m_{n,z}>\eps\},
\qquad E_{\eps,n}:=\cup_{z\in J_{\eps, n}} Q^z.\]

For $z\in I_{\eps,n}$ and $w\in \Z^N$, recalling Riesz rearrangement inequality,
we have
\begin{multline}\label{primo}\int_{Q^z}\int_{Q^w} f_n(x)f_n(y) K(x-y)dxdy \leq \int_{B_{m_{n,w}}}\int_{B_{m_{n,z}}} K^\star(x-y)dxdy\\= \int_{B_{m_{n,w}}}\int_{B_{m_{n,z}}(x)} K^\star(y)dydx\leq m_{n,w}\int_{B_{m_{n,z}}} K^\star(y)dy\leq m_{n,w}\int_{B_\eps} K^\star(y)dy,\end{multline} where the last inequality follows from the fact that $z\in I_{\eps,n}$, and the 
previous inequality from the fact that $K^\star$ is symmetrically decreasing. 

For $R>0$, we compute
\begin{multline*}
\int_{A_{\eps, n}}\int_{\R^N} f_n(x)f_n(y) K(x-y)dxdy\\=\sum_{w\in \Z^N, z\in I_{\eps, n}, |z-w|>R} \int_{Q^z}\int_{Q^w} f_n(x)f_n(y) K(x-y)dxdy\\
+\sum_{w\in \Z^N, z\in I_{\eps, n}, |z-w|\leq R} \int_{Q^z}\int_{Q^w} f_n(x)f_n(y) K(x-y)dxdy.
\end{multline*} 
By \eqref{primo} the second addendum can be bounded as follows 
\begin{multline*} \sum_{w\in \Z^N, z\in I_{\eps, n}, |z-w|\leq R} \int_{Q^z}\int_{Q^w} f_n(x)f_n(y) K(x-y)dxdy\\\leq \sum_{w\in \Z^N} m_{n,w}(2R)^N \int_{B_\eps} K^\star(y)dy
\leq m (2R)^N \int_{B_\eps} K^\star(y)dy. \end{multline*}
On the other hand the first addendum can be bounded as 
\begin{multline*} \sum_{w\in \Z^N, z\in I_{\eps, n}, |z-w|>R} \int_{Q^z}\int_{Q^w} f_n(x)f_n(y) 
K(x-y)dxdy\\\leq  \sum_{z\in I_{\eps, n}}m_{z,n}\int_{|y|>R-\sqrt{N}}    K(y) dy\leq 
m\int_{|y|>R-\sqrt{N}}  K(y) dy.\end{multline*}
Collecting the two estimates, we get
\begin{equation}\label{eqRR}
\int_{A_{\eps, n}}\int_{\R^N} f_n(x)f_n(y) K(x-y)dxdy\\\le
m (2R)^N \int_{B_\eps} K^\star(y)dy + m\int_{|y|>R-\sqrt{N}}  K(y) dy.
\end{equation}
Since $K\in L^1$, hence also $K^\star \in L^1$, we can choose $R=R(\eps)$ in such a way that 
$$
\lim_{\eps\to 0}R(\eps)=+\infty \quad \text{ and }\quad
\lim_{\eps\to 0}R(\eps)^N\int_{B_\eps} K^\star(y)dy =0.
$$
With this choice of $R$, from \eqref{eqRR} we get
\begin{equation}\label{aeps} \int_{A_{\eps, n}}\int_{\R^N} f_n(x)f_n(y) K(x-y)dxdy\leq r(\eps)\end{equation} where $r(\eps)\to 0$ as $\eps\to 0$, uniformly in $n$. 
As a consequence, we obtain
\begin{equation}\label{enuova} 
\mathcal{P}_K(f_n)\geq  m\|K\|_{L^1}- \int_{E_{\eps,n}}\int_{E_{\eps,n}} f_n(x)f_n(y) K(x,y)dxdy- 2r(\eps). 
 \end{equation} 
 
Observe that, due to the fact that $\int_{\R^N} f(x)dx=m$, 
we have $\# J_{\eps,n}\leq m/\eps$. 
Given $w_i, w_j\in J_{\eps,n}$, up to subsequence we get that $|w_i-w_j|\to  c_{i,j}\in\N\cup \{+\infty\}$ as $n\to +\infty$. We consider the 
following sets, for $l=1, \dots, H_\eps$, with $H_\eps\leq \frac{m}{\eps}$ 
\begin{equation}\label{clu} \mathcal{Q}_{\eps,n}^l= \bigcup_{w_{i}\in J_{\eps, n}, \ c_{il}<+\infty} Q^{w_{i}}.\end{equation}
Note that by construction $dist(\mathcal{Q}_{\eps,n}^l, \mathcal{Q}_{\eps,n}^k)\to +\infty$ if $k\neq l$ as $n\to +\infty$. Moreover, always by construction, we get that
$$diam(\mathcal{Q}_{\eps,n}^l)\leq \sum_{i\in \{1, \dots H_\eps\}, c_{il}<\infty} (2c_{il}+2 \sqrt{N})\leq M_\eps,$$ 
where $M_\eps$ does not depend on $n$.

Let $f_n^{l,\eps}:=f_n \chi_{\mathcal{Q}_{\eps,n}^l}$ and let $x_{l,n}$ such that $f_n^{l,\eps}(x_{l,n})>0$. Up to subsequences we can assume that $f_n^{l,\eps}(\cdot+x_{n,l})\stackrel{*}\rightharpoonup f^{l,\eps}$ weakly$^*$  in $L^\infty$, as $n\to +\infty$. 
Observe that the support of $ f_n^{l,\eps}(\cdot+x_{n,l})$ is contained in $B(0,M_\eps)$ for every $n$. 
Moreover, since the functional $\mathcal{P}_K$ is continuous with respect to the tight convergence (see for instance \cite{cdnp}), we get that
\begin{multline*} \lim_n\int_{\R^N}\!\!\int_{\R^N}f_n^{l,\eps}(x+x_{n,l})f_n^{l,\eps}(y+x_{n,l})K(x-y)dxdy\\
= \int_{\R^N}\!\!\int_{\R^N}f^{l,\eps}(x)f^{l,\eps}(y)K(x-y)dxdy.\end{multline*} 
Therefore 
\begin{multline}\label{enuova1} \inf_{\mathcal{A}_m}\mathcal{P}_K= \lim_n\mathcal{P}_K(f_n)\geq m\|K\|_{L^1} - \lim_n  \int_{E_{\eps,n}}\int_{E_{\eps,n}}f_n (x)f_n (y)K(x-y)dxdy-2r(\eps)\\
=  m\|K\|_{L^1} - \sum_{l=1}^{H_\eps} \lim_n  \int_{\R^N}\!\!\int_{\R^N}f_n^{l,\eps}(x+x_{n,l})f_n^{l,\eps}(y+x_{n,l})K(x-y)dxdy-2r(\eps)\\
=  m\|K\|_{L^1} - \sum_{l=1}^{H_\eps} \int_{\R^N}\!\!\int_{\R^N}f^{l,\eps}(x)f^{l,\eps}(y)K(x-y)dxdy-2r(\eps).\end{multline} 

We pass to a subsequence $\eps_k\to 0$ such that $\eps_k$ is decreasing. So $H_{\eps_k}\to H\in(0, +\infty]$. Moreover, we can relabel the sequence in such a way that 
$f_n^{l,\eps_k}$ and then also their limit $f^{l,\eps_k}$ are monotone in $\eps_k$. By monotone convergence $f^{l,\eps_k}\to f^l$ strongly in $L^1$. Moreover if  $m_l=\int_{\R^N} f^l(x)dx$, then $\sum_{l=1}^H m_l=\tilde m\leq m$. Again by continuity of the functional with respect to the
$L^1$-convergence,  from \eqref{enuova1} and from Lemma \ref{suba} we get that
\begin{multline*} \sup_{f\in \mathcal{A}_m}\int_{\R^N}\!\!\int_{\R^N} f(x)f(y)K(x-y)dxdy\\
= \lim_n\int_{\R^N}\!\!\int_{\R^N} f_n(x)f_n(y)K(x-y)dxdy)\leq  \sum_{l=1}^{H } \int_{\R^N}\!\!\int_{\R^N}f^{l}(x)f^{l}(y)K(x-y)dxdy\\
\leq \sum_{l=1}^{H }\sup_{\mathcal{A}_{m_l}}\int_{\R^N}\!\!\int_{\R^N} f(x)f(y)K(x-y)dxdy\leq \sup_{\mathcal{A}_{\tilde m}}\int_{\R^N}\!\!\int_{\R^N} f(x)f(y)K(x-y)dxdy\\\leq  \sup_{\mathcal{A}_{m}}\int_{\R^N}\!\!\int_{\R^N} f(x)f(y)K(x-y)dxdy.
  \end{multline*}
  Therefore the previous are all equalities, and  $f^l$  is a minimizer  of $\mathcal{P}_K$ 
  in $\mathcal{A}_{m_l}$ for all $l$'s. 
  In particular, recalling again Lemma \ref{suba}, 
  we get that $H=1$, and  $f^1$ 
  is a minimizer  of $\mathcal{P}_K$  in $\mathcal{A}_{m}$. 
\end{proof} 

We finally show  that, under a further condition on $K$, the isoperimetric problem \eqref{isopb} admits a solution. 
 
\begin{theorem} \label{existenceiso} 
Assume  that for a.e. $x\in\R^N$ there exists $\eps_x>0$ such that  for all $\eps<\eps_x$ 
\begin{equation}\label{pos} 
\int_{B(0,2\eps)} |B(0,\eps)\cap B(z,\eps)| ( K(z)-K(x+z)) dz>0.
 \end{equation} 
Then, for every $m>0$ there exists a compact set $E\subseteq \R^N$ such that $|E|=m$ and  
$E$ solves the isoperimetric problem \eqref{isopb}.
\end{theorem}

\begin{proof}  By Theorem \ref{existence} there exists at least one $f\in \mathcal{A}_m$ which solves the minimization problem \eqref{min}. 
Moreover the support of $f$ is compact due to  Proposition \ref{compact}. 

Assume by contradiction that $f$  is not a characteristic function. Then there exist $\bar x\neq \bar y$ Lebesgue points of $f$  
such that $0<f(\bar x), f(\bar y)<1$. Let $\eps<\frac{1}{2} |\bar x-\bar y|$ and define the function $\xi(x):=\chi_{B(\bar x,\eps)}-\chi_{B(\bar y,\eps)}$. Therefore, by the  second variation formula \eqref{sv}, for every $\eps<\frac{1}{2}|\bar x-\bar y|$ we get 
\begin{multline*}  0\geq \int_{\R^N}\!\!\int_{\R^N} \xi(y)\xi(x)K(x-y)dxdy\\ =2\int_{B(0,\eps)}\int_{B(0,\eps)} (K(x-y)-K(\bar x-\bar y-(x-y))dxdy\\= 2\int_{B(0,2\eps)}|B(0,\eps)\cap B(z,\eps)|(K(z)-K(\bar x-\bar y+ z))dz, \end{multline*} 
which contradicts \eqref{pos}, and concludes the proof.
\end{proof} 

\begin{remark}\upshape A sufficient condition for \eqref{pos} to hold is that 
\begin{equation}\label{max} 
\liminf_{z\to 0}\,(K(z)-K(z+x))>0 \qquad \text{for a.e. }x\in \R^N.
\end{equation} 
In particular this condition is always verified if $K$ is positive definite, that is, 
\[ \begin{cases} \int_{\R^N}\!\!\int_{\R^N} \phi(x)\phi(y)K(x-y)dxdy\geq 0  & \forall \phi \in L^1(\R^N)\\ 
 \int_{\R^N}\!\!\int_{\R^N} \phi(x)\phi(y)K(x-y)dxdy= 0 & \text{ iff }\phi\equiv 0.\end{cases}\]
\end{remark}



\begin{thebibliography}{99}

\bibitem{afp} \newblock L. Ambrosio, N. Fusco, D. Pallara.
\newblock  \textit{Functions of Bounded Variation and Free Discontinuity Problems}
\newblock  Oxford Mathematical Monographs, 2000. 

\bibitem{crs} \newblock L. A. Caffarelli, J.-M. Roquejoffre, O. Savin.
\newblock Nonlocal minimal surfaces. 
\newblock \emph{Comm. Pure Appl. Math.}, 63 (2010), no. 9, 1111--1144. 
 
\bibitem{cdnv}
\newblock A. Cesaroni, S. Dipierro, M. Novaga, E. Valdinoci.
\newblock  Minimizers for nonlocal perimeters of Minkowski type.
\newblock  \emph{Arxiv preprint} 2017, https://arxiv.org/abs/1704.03195.

\bibitem{cn}
\newblock A. Cesaroni, M. Novaga.
\newblock  Volume constrained minimizers of the fractional perimeter with a potential energy.
\newblock  \emph{Discrete Contin. Dyn. Syst. S}, 4 (2017), no. 10, 715--727.

\bibitem{cmp1}
\newblock A. Chambolle, M. Morini, M. Ponsiglione.
\newblock  Nonlocal curvature flows. 
\newblock \emph{Arch. Ration. Mech. Anal.}, 218 (2015), no. 3, 1263--1329. 

\bibitem{cdnp} 
\newblock M. Cicalese, L. De Luca, M. Novaga, M. Ponsiglione.
\newblock Ground states of a two phase model with cross and self attractive interactions.
\newblock \emph{SIAM J. Math. Anal.}, 48 (2016), no. 5, 3412--3443.

\bibitem{csv} 
\newblock E. Cinti, J. Serra, E. Valdinoci.
\newblock Quantitative flatness results and BV-estimates for stable nonlocal minimal surfaces.
\newblock \emph{Arxiv preprint} 2016, https://arxiv.org/abs/1602.00540.

\bibitem{dnrv} 
\newblock A. Di Castro, M. Novaga, B. Ruffini, E.Valdinoci.
\newblock Nonlocal quantitative isoperimetric inequalities. 
\newblock \emph{Calc. Var. Partial Differential Equations}, 54 (2015), no. 3, 2421--2464. 

\bibitem{ffmmm}
\newblock A. Figalli,  N. Fusco, F. Maggi, V. Millot, M. Morini.
\newblock Isoperimetry and stability properties of balls with respect to nonlocal energies. 
\newblock \emph{Comm. Math. Phys.}, 336 (2015), no. 1, 441--507.


\bibitem{gn}
\newblock M. Goldman, M. Novaga.
\newblock Volume-constrained minimizers for the prescribed curvature problem in periodic media. 
\newblock  \emph{Calc. Var. Partial Differential Equations}, 44 (2012), no. 3-4, 297--318.
 
\bibitem{l}
\newblock M. Ludwig.
\newblock Anisotropic fractional perimeters. 
\newblock \emph{J. Differential Geom.}, 96 (2014), no. 1, 77--93.

\bibitem{maggibook}
\newblock F. Maggi. 
\newblock \textit{Sets of finite perimeter and geometric variational problems. In: An introduction to 
Geometric Measure Theory.}
\newblock Cambridge Studies in Adavanced Mathematics, vol. 135, Cambridge University Press, Cambridge, 2012.


\bibitem{m} 
\newblock V. Maz'ya.
\newblock  { Lectures on isoperimetric and isocapacitary inequalities in the theory of Sobolev spaces. }
\newblock Contemp. Math., 338, Amer. Math. Soc., Providence, RI, 2003. 

\bibitem{riesz} 
\newblock F. Riesz. 
\newblock {Sur une in\'egalit\'e int\'egrale. Journ.} 
\newblock \emph{London Math. Soc.}, 5 (1930),162--168.

\bibitem{v}
\newblock A. Visintin.
\newblock Generalized coarea formula and fractal sets. 
\newblock \emph{Japan J. Indust. Appl. Math.} 8 (1991), 175--201.

\end{thebibliography}
\end{document}